\documentclass[12pt,leqno]{amsart}
\usepackage{amsmath,amssymb,amsfonts}
\usepackage{eucal,enumerate}
\usepackage{xypic,graphicx,psfrag}
\usepackage{mathrsfs}
\usepackage{hyperref}

\usepackage{color} 

\usepackage{tikz}
\usetikzlibrary{decorations.pathreplacing}
\usepackage{scalefnt}

\newcommand\sectionpage\newpage

\newtheorem{lem}{Lemma}[section]
\newtheorem{cor}{Corollary}[section]
\newtheorem{prop}{Proposition}[section]
\newtheorem{thm}{Theorem}[section]

\theoremstyle{definition}
\newtheorem{conj}{Conjecture}[section]
\newtheorem{prob}{Problem}[section]

\numberwithin{equation}{section}

\renewcommand\mod{\, \operatorname{mod}\, }
\renewcommand{\phi}{\varphi} 
\renewcommand{\epsilon}{\varepsilon}

\newcommand\eset{\varnothing}

\newcommand\codim{\operatorname{codim}}

\newcommand\cA{\mathscr{A}}		
\newcommand\cB{\mathcal{B}}
\newcommand\cBo{{\cB^\circ}}
\renewcommand\cH{\mathcal{H}}	

\renewcommand\cL{\mathscr{L}}	
\newcommand\cP{\mathcal{P}}

\newcommand\cU{\mathcal{U}}	
\newcommand\tcU{\widetilde\cU}

\newcommand\cW{\mathcal{W}}

\newcommand\bbR{\mathbb{R}}
\newcommand\bbZ{\mathbb{Z}}

\newcommand\pN{\mathbb N}	
\newcommand\pP{\mathbb P}	
\newcommand\pQ{\mathbb Q}	

\newcommand\bL{\mathbf{L}}

\newcommand\lcm{\operatorname{lcm}}

\newcommand\lB{l_\cB}

\newcommand\M{\mathbf{M}}
\newcommand\Kot{Kot\v{e}\v{s}ovec}
\newcommand\cube{[0,1]^{2q}}
\newcommand\ocube{(0,1)^{2q}}
\newcommand\osquare{(0,1)^{2}}
\newcommand\barn{{\bar n}}
\newcommand\hatc{{\hat c}}
\newcommand\hatd{{\hat d}}
\newcommand\Aut{\operatorname{Aut}}
\renewcommand\pmod[1]{\;(\textrm{mod}\;#1)}
\newcommand\nhatd{\barn}

\newcommand\xrefcomment[1]{}	
\newcommand\Ptwopiecetypes{I.5.6\xrefcomment{P:2piecetypes}}

\newcommand\Tgammapoly{I.4.2\xrefcomment{T:gammapoly}}
\newcommand\Eiopmu{(I.2.1)\xrefcomment{E:iopmu}}
\newcommand\Ttwomovetypes{I.5.8\xrefcomment{T:2movetypes}}
\newcommand\Tformula{I.4.1\xrefcomment{T:formula}}
\newcommand\Ttypenumber{I.5.3\xrefcomment{T:typenumber}}
\newcommand\Phvdiag{III.3.1\xrefcomment{P:hvdiag}}	
\newcommand\N{IV.6\xrefcomment{N}}		

\allowdisplaybreaks

\begin{document}

\pagestyle{myheadings}
\markleft{Chaiken, Hanusa, and Zaslavsky \qquad\qquad \today}
\markright{A $q$-Queens Problem. II.  The Square Board \quad \today}

\title{A $q$-Queens Problem. 
II.  The Square Board \\[10pt]  
\today}

\author{Seth Chaiken}
\address{Computer Science Department\\ The University at Albany (SUNY)\\ Albany, NY 12222, U.S.A.}
\email{\tt sdc@cs.albany.edu}

\author{Christopher R.\ H.\ Hanusa}
\address{Department of Mathematics \\ Queens College (CUNY) \\ 65-30 Kissena Blvd. \\ Queens, NY 11367-1597, U.S.A.}
\email{\tt chanusa@qc.cuny.edu}

\author{Thomas Zaslavsky}
\address{Department of Mathematical Sciences\\ Binghamton University (SUNY)\\ Binghamton, NY 13902-6000, U.S.A.}
\email{\tt zaslav@math.binghamton.edu}

\begin{abstract}
We apply to the $n\times n$ chessboard the counting theory from Part~I for nonattacking placements of chess pieces with unbounded straight-line moves, such as the queen.  Part~I showed that the number of ways to place $q$ identical nonattacking pieces is given by a quasipolynomial function of $n$ of degree $2q$, whose coefficients are (essentially) polynomials in $q$ that depend cyclically on $n$.  

Here we study the periods of the quasipolynomial and its coefficients, which are bounded by functions, not well understood, of the piece's move directions, and we develop exact formulas for the very highest coefficients.  
The coefficients of the three highest powers of $n$ do not vary with $n$.  On the other hand, we present simple pieces for which the fourth coefficient varies periodically.  We develop detailed properties of counting quasipolynomials that will be applied in sequels to partial queens, whose moves are subsets of those of the queen, and the nightrider, whose moves are extended knight's moves.  

We conclude with the first, though strange, formula for the classical $n$-Queens Problem and with several conjectures and open problems.
\end{abstract}



\thanks{The outer authors thank the very hospitable Isaac Newton Institute for facilitating their work on this project.  The inner author gratefully acknowledges support from PSC-CUNY Research Awards PSCOOC-40-124, PSCREG-41-303, TRADA-42-115, TRADA-43-127, and TRADA-44-168.}

\maketitle

\setcounter{tocdepth}{4}
\tableofcontents

\section{Introduction}\label{intro}

The well known $n$-Queens Problem asks for the number ways to place $n$ nonattacking queens on an $n\times n$\label{d:n} chessboard.  No general formula is known; the answer has been computed separately for each small value of $n$.  In this article, the second of a series~\cite{QQ}, we treat a natural generalization, the $q$-Queens Problem, in which both the number of queens, $q$, and the size of the board, $n$, vary independently.  In Part~I we developed a general theory for arbitrary convex-polygonal boards with rational vertices and any chess piece $\pP$ with unbounded straight-line moves (known as a ``rider''); this includes the queen, rook, and bishop as well as fairy chess pieces such as the nightrider, whose moves are those of a knight extended to any distance.  The main result of Part~I  (Theorem~\Tformula) is that the answer is a quasipolynomial function of $n$, which means it is given by a cyclically repeating sequence of polynomials, and the coefficients of powers of $n$ are essentially polynomial functions of $q$; in fact, we found a complicated general coefficient formula.  We deduced this from the theory of inside-out polytopes, which is an extension of Ehrhart's theory of counting lattice points in convex polytopes.  

Here and in later parts we treat the square board.  The self-similarity of its interior lattice points permits us to find explicit formulas for the two highest coefficients and partial formulas for the others, for arbitrary riders.  Setting $q=n$ we obtain the first known formula for the $n$-Queens Problem.  

In Parts~III through V we specialize to specific pieces.  Part~III applies the theory of Parts~I and II to ``partial queens'', which have a subset of the queen's moves.  We found partial queens easier to work with than other pieces with equally many moves; we also think they make good test cases for conjectures.  The main examples of partial queens are the bishop, queen, and rook, which we treat in detail along with the nightrider in Parts~IV and V.

\medskip
Part~I serves as a general introduction to the series.  Its introduction gives a fuller description of the background of our research, including the valuable hints obtained from the formulas collected and developed by V\'aclav \Kot\ (see~\cite{ChMath}).  
We define the necessary terminology and notation from Part~I; we provide a dictionary of notation to assist the reader (and authors).  

Briefly (see Section~\ref{prelim} for further detail), the board consists of the integral points in the dilate $(n+1)\cBo = (n+1)\osquare$\label{d:n+1} of the open unit square.  
The associated polytope is $\cP = \cB^q = \cube$ with interior $\cP^\circ = \cBo^q = \ocube$.  The inside-out polytope is $(\cube,\cA_\pP)$\label{d:iop} where $\cA_\pP$ is an arrangement of hyperplanes determined by the moves of the piece $\pP$.\label{d:P}  Each hyperplane is the kernel of an expression involving two pieces.  An \emph{intersection subspace} is an intersection of a subset of $\cA_\pP$; the lattice of all intersection subspaces (ordered by reverse inclusion) is denoted by $\cL(\cA_\pP)$,\label{d:L} or by $\cL(\cA_\pP^q)$ when we wish to emphasize the value of $q$, and $\mu$ denotes its M\"obius function.\label{d:mu}  

The moves of $\pP$ are all integral multiples of vectors in a nonempty set $\M$ of non-zero, non-parallel integral vectors $m_r=(c_r,d_r) \in \bbR^2$ reduced to lowest terms (that is, $c_r$ and $d_r$ are relatively prime).  The counting function $u_\pP(q;n)$\label{d:indistattacks} is defined as the number of nonattacking configurations of $q$ indistinguishable copies of $\pP$ on the $n \times n$ board, and $o_\pP(q;n)$\label{d:distattacks} is the number of such configurations of $q$ distinguishable copies.  By Theorem~\Tformula\ $u_\pP(q;n)$ is a quasipolynomial function of $n$, which we expand as 
$$
u_\pP(q;n) = \gamma_0(n) n^{2q} + \gamma_1(n) n^{2q-1} + \gamma_2(n) n^{2q-2} + \cdots + \gamma_{2q}(n) n^0.  
$$
By Ehrhart theory the leading coefficient is $\gamma_0(n) = 1/q!$\label{d:gamma} and the period of $u_\pP(q;n)$ is a divisor of the denominator $D(\cube,\cA_\pP)$, defined as the least common denominator of all coordinates of all vertices of $(\cube,\cA_\pP)$.  (The denominator of a polytope alone is defined similarly.)  
By Theorem~\Ttypenumber\ the number of unlabelled combinatorial types of nonattacking configuration equals $u_\pP(q;-1)$.  

A trivial observation is that 
$u_\pP(1;n) = n^2$
for any piece, and with one piece there is (of course) one combinatorial type.  Theorem~\ref{T:u2P} gives a complete solution for $u_\pP(2;n)$, the counting function for two copies of an arbitrary rider piece.  
General formulas for $u_\pP(q;n)$ when $q\geq 3$ are difficult.  

A computational approach to finding $u_\pP(q;n)$ explicitly for a particular piece is to evaluate it at enough small values of $n$ by counting the non-attacking configurations, bounding the period $p$\label{d:p} somehow (possibly by bounding the denominator $D(\cP,\cA_\pP)$), and using that information to interpolate the coefficients of the $p$ constituent polynomials.  To get a confirmed answer, $2pq$ values of $u_\pP(q;n)$ must be computed.  (See Section~\ref{recurrence} for more detail.)  This method becomes hard for most problems because it involves a daunting amount of computation if the period or its best known bound is large, as is usually the case.  That is why we think it important to find good bounds on the period.  We find some periods here; we propose related conjectures and problems in Section~\ref{bounds}.

Any preliminary information about $u_\pP(q;n)$ can reduce the number of required values.  Our main result, Theorem~\ref{T:gammapolysquare}, reduces that number by $2p-1$ by giving a simple formula for the second coefficient, $\gamma_1$, and proving that $\gamma_2$ is independent of $n$.  (We reduce it by two more by noting that $u_\pP(q;0) = 0$ and $u_\pP(q;1) = \delta_{q1}$, the Kronecker delta.\label{d:KD})
The proof of Theorem~\ref{T:gammapolysquare} depends on the structure of the subspace Ehrhart functions developed in Section~\ref{ehrhartcoeffs}.  In Section~\ref{1move}, through an explicit construction involving pieces with one move direction, we show that $\gamma_3$ may depend on $n$, in contrast to the constancy of $\gamma_0$, $\gamma_1$, and $\gamma_2$. 

Our formula for the $n$-Queens Problem (Section~\ref{nqueens}) is an immediate corollary of Theorem~\ref{T:gammapolysquare}.  It is not clear how practical this formula is, as it has infinitely many terms (though only finitely many for each $q$) that become harder and harder to evaluate, but it is precise and complete and shows a clear structure from which it may be possible to deduce interesting consequences---which we do not attempt to do here.

We conclude with open problems, of which one, dealing with recurrences satisfied by $u_\pP(q;n)$ for fixed $q$ (Section~\ref{recurrence}), promises to be superbly important.

\sectionpage
\section{Preliminaries}\label{prelim}

We adopt the concise notation $[n] := \{1,\ldots,n\}$\label{d:[n]} so that the set of points representing the squares of an $n \times n$ chessboard is
$$
[n]^2 = (n+1)\osquare \cap \bbZ^2.\label{d:[n]2}
$$ 

The hyperplane representing an attack between pieces $\pP_i$ and $\pP_j$, located at $z_i=(x_i,y_i)$ and $z_j=(x_j,y_j)$, in the direction of the basic move $m=(c,d)$, i.e.\ along slope $d/c$, is $\cH^{d/c}_{ij}=\cH^{d/c}_{ji}$.\label{d:cH}  Its defining equation is $d(x_j-x_i)=c(y_j-y_i)$.  The precise definition of the move arrangement is 
$$
\cA_\pP := \{ \cH^{d/c}_{ij} : (c,d)\in\M, \ 1\leq i < j \leq q \}.
\label{d:cA}
$$

The \emph{Ehrhart quasipolynomial} of a rational convex polytope $\cP$ is $E_\cP(n+1):=$\label{d:Ehr} the number of integral lattice points in the dilate $(n+1)\cP$.  The \emph{open Ehrhart quasipolynomial} of $\cP$ is $E_\cP^\circ(n+1):=$\label{d:E} the number of integral points in the interior $(n+1)\cP^\circ$.  The open Ehrhart quasipolynomial of an inside-out polytope $(\cP,\cA)$ is $E_{\cP,\cA}^\circ(n+1):=$\label{d:Eiop} the number of integral points in $(n+1)\cP^\circ$ that are not in any hyperplane of $\cA$.

We introduce a short notation for a frequently recurring quantity.  For $\cU \in \cL(\cA_\pP)$,\label{d:cU} let $\kappa$ be the number of pieces involved in the equations that determine $\cU$ and let $\tcU$ be the essential part of $\cU$, i.e., the subspace of $\bbR^{2\kappa}$ that satisfies the same attack equations as $\cU$.  The ``reduced'' open Ehrhart quasipolynomial of $\cU$ is
$$
\alpha(\cU;n):=E_{(0,1)^{2\kappa}\cap\tcU}(n+1).
\label{d:acU}
$$ 
The actual open Ehrhart quasipolynomial is $\alpha(\cU;n)n^{2q-2\kappa}$.  

Subspaces $\cU_1,\cU_2 \in \cL(\cA_\pP)$ are \emph{isomorphic} if the $q$ copies of $\pP$ can be relabelled so that one subspace becomes the other.  The \emph{type} $[\cU]$\label{d:type} of $\cU$ is its isomorphism class.  For instance, $[\cH_{12}^{2/1}]=[\cH_{25}^{2/1}]\neq[\cH_{12}^{1/2}]$.  The relabelling is an \emph{isomorphism}.  The group of automorphisms of $\cU$ is denoted by $\Aut\cU$.\label{d:Aut}

\sectionpage\section{Two Pieces}\label{1or2}

We examine an exceedingly small number of pieces, i.e., $q=2$.  There is a (relatively) simple way to calculate $u_\pP(2;n)$.  Define $a_\pP(2;n)$\label{d:aP} to be the number of attacking configurations of two labelled pieces $\pP$ (which may occupy the same position; that is considered attacking).  Then 
\begin{equation}
u_\pP(2;n) = \frac{1}{2!} o_\pP(2;n) = \frac{1}{2} \big[ n^4 - a_\pP(2;n) \big].
\label{E:2}
\end{equation}
Finding $a_\pP(2;n)$ is easy in principle although nontrivial in detail.  (See Equation~\eqref{E:alphasum}.)  To begin with, consider a move $(c,d) \in \M$,\label{d:cd2} whose slope is the rational fraction $d/c$, and let 
$$
l^{d/c}(b) := \text{the line in $\bbR^2$ with slope $d/c$ and $y$-intercept $b$}.
\label{ldcb}
$$
We allow $d/c=1/0 = \infty$, in which case $b$ is instead the $x$-intercept.  Define 
$$
\lB^{d/c}(b) := l^{d/c}(b) \cap [n]^2,
$$
the set of positions on the $n \times n$ board $[n]^2$ that lie on the line $l^{d/c}(b)$.  The multiset of line sizes,
$$
\bL^{d/c}(n):=\big\{|\lB^{d/c}(b)| : \lB^{d/c}(b)\neq\eset\big\},\label{d:linesizes}
$$ 
is finite and the sum of its entries is $n^2$.  
We need to know the the exact contents of $\bL^{d/c}(n)$.  Two cases are elementary:
\begin{equation*}
\begin{aligned}
\bL^{0/1}(n) = \bL^{1/0}(n) &= \{ n^{n} \}, \\
\bL^{1/1}(n) = \bL^{-1/1}(n) &= \{ 1^2, 2^2, \ldots, (n-1)^2, n^1 \}.
\end{aligned}
\end{equation*}

\begin{lem}\label{p:bL(n)}
Assume $0 < c \leq d$ are relatively prime integers.  Let $\barn:=(n\mod d)$.\label{d:barn}  
The multiplicities of line sizes in $\bL^{d/c}(n)$ are as in the following table:
\begin{equation*}
\begin{tabular}{l@{\qquad}c@{\quad}c@{\quad}c}
Line size	& $1\leq l < \lfloor\frac{n}{d}\rfloor$	& $\lfloor\frac{n}{d}\rfloor$		& $\lfloor\frac{n}{d}\rfloor+1$	\\[6pt]
Multiplicity& $2cd$		& $(d-\barn)(n-c\lfloor\frac{n}{d}\rfloor)+c(\barn+d)$	& $\barn(n-c\lfloor\frac{n}{d}\rfloor)$
\end{tabular}
\end{equation*}
\end{lem}

\begin{proof}
Let $\delta := \lfloor n/d\rfloor = (n-\barn)/d$;\label{d:n/d} note that $\delta \leq \lfloor n/c\rfloor$.  

Each nonempty line $\lB^{d/c}(b)$ has a lowest point 
$$(x,y) \in Z := \{(x,y) \in [n]^2: x \leq c \text{ or } y \leq d\},\label{d:Z}$$ 
and conversely, each point in $Z$ is the lowest point of a different line $\lB^{d/c}(b)$.  If we rename the line $\lB(x,y)$, the naming is unique and the points on $\lB(x,y)$ are the points of the form $(x,y)+k(c,d)$ for $k=0,\ldots,\bar k$, where $\bar k$ is the largest integer such that $(x,y)+\bar k(c,d) \in [n]^2$.  Solving this last restriction for $\bar k$, we find that 
$$x\leq n \implies \bar k \leq (n-x)/c \text{ and } y \leq n \implies \bar k \leq (n-y)/d.$$
Hence,  $\bar k = \min\big( \lfloor(n-x)/c\rfloor, \lfloor(n-y)/d\rfloor ).$

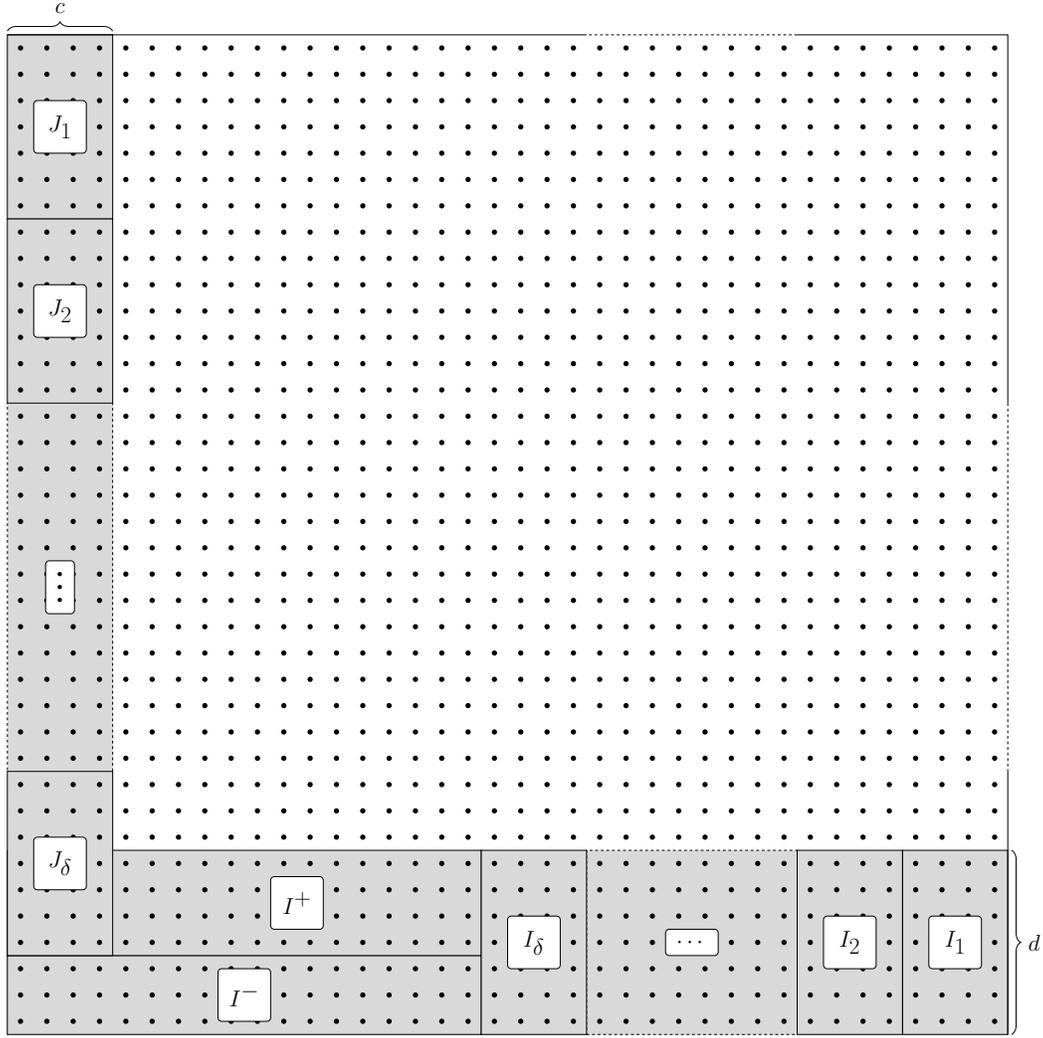
\begin{figure}[htbp]
\begin{center}
\scalebox{0.35}{\scalefont{6}
\begin{tikzpicture}


        \draw[dashed,fill=black!15] (30.5,0.5) rectangle (22.5,7.5);  
        \draw[dashed,fill=black!15] (0.5,24.5) rectangle (4.5,10.5);  
        \draw[thick,fill=black!15] (0.5,0.5) rectangle (18.5,3.5);  
        \draw[thick,fill=black!15] (0.5,3.5) rectangle (18.5,7.5);  

        \draw[thick,fill=black!15] (0.5,38.5) rectangle (4.5,31.5);  
        \draw[thick,fill=black!15] (0.5,31.5) rectangle (4.5,24.5);  
        \draw[thick,fill=black!15] (0.5,10.5) rectangle (4.5,3.5);  
        \draw[thick,fill=black!15] (38.5,0.5) rectangle (34.5,7.5);  
        \draw[thick,fill=black!15] (34.5,0.5) rectangle (30.5,7.5);  
        \draw[thick,fill=black!15] (22.5,0.5) rectangle (18.5,7.5);  
        

        \draw[dashed] (30.5,38.5) -- (22.5,38.5); 
        \draw[dashed] (38.5,24.5) -- (38.5,10.5); 
        \draw[thick] (38.5,24.5) -- (38.5,38.5) -- (30.5,38.5);
        \draw[thick] (22.5,38.5) -- (0.5,38.5);
        \draw[thick] (38.5,10.5) -- (38.5,0.5);

        
        \foreach \x in {1,...,38} {
            \foreach \y in {1,...,38} {
                \node at (\x,\y) [circle,inner sep=0pt,fill=black,minimum size=0.2cm] {};
            }
        }


        \node[rectangle,draw,rounded corners,thick,minimum width=2cm,minimum height=2cm,fill=white]  at (11.5,5.5) {$I^+$}; 
        \node[rectangle,draw,rounded corners,thick,minimum width=2cm,minimum height=2cm,fill=white]  at (9.5,2) {$I^-$}; 
        \node[rectangle,draw,rounded corners,thick,minimum width=2cm,minimum height=2cm,fill=white]  at (36.5,4) {$I_1$}; 
        \node[rectangle,draw,rounded corners,thick,minimum width=2cm,minimum height=2cm,fill=white]  at (32.5,4) {$I_2$}; 
        \node[rectangle,draw,rounded corners,thick,minimum width=2cm,minimum height=1cm,fill=white]  at (26.5,4) {$\cdots$};
        \node[rectangle,draw,rounded corners,thick,minimum width=2cm,minimum height=2cm,fill=white]  at (20.5,4) {$I_\delta$};
        \node[rectangle,draw,rounded corners,thick,minimum width=2cm,minimum height=2cm,fill=white]  at (2.5,35) {$J_1$}; 
        \node[rectangle,draw,rounded corners,thick,minimum width=2cm,minimum height=2cm,fill=white]  at (2.5,28) {$J_2$}; 
        \node[rectangle,draw,rounded corners,thick,minimum width=1cm,minimum height=2cm,fill=white]  at (2.5,17.5) {\normalsize\begin{tabular}{c}$\bullet$\\$\bullet$\\$\bullet$\\\end{tabular}}; 
        \node[rectangle,draw,rounded corners,thick,minimum width=2cm,minimum height=2cm,fill=white]  at (2.5,7) {$J_\delta$};

\draw [decorate,decoration={brace,amplitude=10pt,mirror,raise=4pt},thick,yshift=0pt] (38.5,0.5) -- (38.5,7.5);
        \node[rectangle] at (39.5,4) {$d$}; 
\draw [decorate,decoration={brace,amplitude=10pt,raise=4pt},thick,yshift=0pt] (0.5,38.5) -- (4.5,38.5);
        \node[rectangle] at (2.5,39.5) {$c$};

\end{tikzpicture}
}
\caption{The division of $[n]^2$ into the shaded border region $Z$, with its subdividing rectangles, and the remainder of the board.  (The illustration shows the case where $I_\delta$ and $J_\delta$ do not overlap.)}
\label{fig:grid}
\end{center}
\end{figure}

In order to calculate the cardinality of a line $\lB(x,y)$ for $(x,y)\in Z$ we pick out special subrectangles in $[n]^2$ (illustrated in Figure~\ref{fig:grid}).  First are the lower and left borders:
$$
I := \{ (x,y) \in [n]^2 : y \leq d \}, \quad 
J := \{ (x,y) \in [n]^2 : x \leq c \}. 
\label{d:IJ}
$$
Define new $c \times d$ rectangles on the bottom edge, from right to left, 
$$
I_i := \{ (x,y)\in I : n-ci < x \leq n-c(i-1) \}  \quad\text{ for }  i=1,\ldots,\delta,
$$
and on the left edge, from the top down, 
$$
J_j := \{ (x,y)\in J : n-dj < y \leq n-d(j-1) \}  \quad\text{ for } j=1,\ldots,\delta .
$$
Thus, $I_1$ occupies the bottom right corner of $[n]^2$ and $J_1$ occupies the top left corner of $[n]^2$.  
Also, $J_\delta$ occupies the upper half of the left end of $I$.  

Then, subdivide the remainder of $Z$ (that is, the part of $I$ to the left of $I_\delta$ and not in $J_\delta$) into lower and upper halves:
\begin{align*}
I^- &:= \{ (x,y) \in I : y \leq \barn, \ 1 \leq x \leq n-c\delta \}, \\
I^+ &:= \{ (x,y) \in I : y > \barn, \ c+1 \leq x  \leq n-c\delta \}.
\end{align*}

There is a critical value of $x$, namely, $n-c\delta,$ such that $\bar k = \lfloor(n-x)/c\rfloor$ if $x > n-c\delta$ and $\bar k = \lfloor(n-y)/d\rfloor$ if $x \leq n-c\delta$.  Hence, we know the size of any line by the formula 
\begin{equation}
|\lB(x,y)| = \bar k + 1 = \begin{cases}
\delta+1,	&\text{ if } (x,y) \in I^-,\\
\delta,	&\text{ if } (x,y) \in I^+,\\
i,	&\text{ if } (x,y) \in I_i \text{ for } i \leq \delta, \\
j,	&\text{ if } (x,y) \in J_j \text{ for } j \leq \delta.
\end{cases}
\label{E:linesizes}
\end{equation}
From Equation~\eqref{E:linesizes} we can write down the multiplicities of all line sizes in the multiset $\bL^{d/c}(n)$ by counting the base points $(x,y)$ in each case.  We obtain the multiplicities stated in the lemma.

The rectangles $I_\delta$ and $J_\delta$ do not overlap if and only if the left end of $I_\delta$, at $x=n-c\delta+1$, is to the right of the right edge of $J_\delta$, at $x=c$; that is, if and only if $n-(c+1)\delta \geq 0$; equivalently, $\delta = \lfloor n/c \rfloor$.  As the width of $I^+$ is exactly $(n-c\delta)-c$, if there is overlap then $I^+$ is the overlap and has negative width, so our computation subtracts exactly the amount necessary to correct for double counting of the lines based at $(x,y) \in I_\delta \cap J_\delta$.  (In this case $I^+ := \{ (x,y) \in I : y > \barn, \ c \geq x  > n-c\delta \}$.)  Thus, our formula works whether or not overlap occurs.
\end{proof}

Now, define $\alpha^{d/c}(n)$\label{d:adc} to be the number of ordered pairs of positions that attack each other along slope $d/c$.  Thus, 
$$
\alpha^{d/c}(n) := \alpha(\cH^{d/c}_{12};n) = E_{(0,1)^4\cap\cH_{12}^{d/c}}(n+1),
$$  
the open Ehrhart quasipolynomial of the subpolytope $[0,1]^4 \cap \cH_{12}^{d/c}$ of $[0,1]^4$ that satisfies the equation of attack, $(z_2-z_1) \cdot (d,-c) = 0$, of $\cH_{12}^{d/c}$.  Counting attacking pairs of positions shows that 
$$
\alpha^{d/c}(n) = \sum_{l \in \bL^{d/c}(n)} l^2.
$$  
The subpolytope is 3-dimensional so the degree of $\alpha^{d/c}(n)$ is 3; therefore its leading coefficient is the relative volume of $[0,1]^4 \cap \cH_{12}^{d/c}$.
Similarly, the number of ordered triples that are collinear along slope $d/c$ is  
$$
\beta^{d/c}(n) := \alpha(\cW^{\,d/c}_{123};n) = E_{(0,1)^6\cap\cW^{\,d/c}_{123}}(n+1) = \sum_{l \in \bL^{d/c}(n)} l^3,
\label{d:bdc}
$$
where $\cW^{\,d/c}_{ijk\ldots} := \cH^{d/c}_{ij} \cap \cH^{d/c}_{jk} \cap \cdots$,\label{d:Wdc} the subspace in which $\pP_i,\pP_j,\pP_k,\ldots$ attack each other along slope $d/c$.
The leading coefficient is the relative volume of $[0,1]^6\cap\cW^{\,d/c}_{123}$.

The values of $\alpha^{d/c}(n)$ et al.\ have to be computed for each slope.  Two easy examples are
\begin{align}
\qquad
\alpha^{0/1}(n) = \alpha^{1/0}(n) &= n^3, & 
\alpha^{\pm1/1}(n) = \sum_{i=1}^n i^2 + \sum_{i=1}^{n-1} i^2 &= \frac{2n^3+n}{3}. 
\label{E:attacklinesQ}
\end{align}
and 
\begin{align}
\beta^{0/1}(n) = \beta^{1/0}(n) &= n^4, & 
\beta^{\pm1/1}(n) &= \frac{n^4+n^2}{2}. 
\label{E:attacklines3Q}
\end{align}
There are general formulas.

\begin{prop}\label{p:attackdc}
For relatively prime integers $c\geq0$ and $d>0$ with $c \leq d$, let $\barn:=(n\mod d) \in \{0,1,\ldots,d-1\}$.  
The number of ordered pairs of positions that attack each other along lines of slope $d/c$ is 
\begin{equation}
\begin{aligned}
\alpha^{d/c}(n) =&\ \bigg\{ \frac{3d-c}{3d^2}\, n^3 + \frac{c}{3}\, n \bigg\} 
+ \frac{\barn(d-\barn)}{d^2} \bigg\{ (d-c) n - \frac{c(d-2\barn)}{3} \bigg\}.
\end{aligned}
\label{E:alpha d/c}
\end{equation}
The period of this quasipolynomial is $d$.  

The number of ordered triples of positions that attack each other along a single line of slope $d/c$ is
\begin{equation}
\begin{aligned}
\beta^{d/c}(n) &\;= \bigg\{ \frac{2d-c}{2d^3}\, n^4 + \frac{c}{2d}\, n^2 \bigg\} 
\\&\qquad 
+ \frac{\barn(d-\barn)}{d^3} \bigg\{ 3(d-c) n^2 - (d-2c)(d-2\barn) n + \frac{3c\barn(d-\barn)}{2} \bigg\}.
\end{aligned}
\label{E:alpha d/c 3}
\end{equation}
The period of this quasipolynomial is $d$.  
\end{prop}

Each of these quasipolynomials has an invariant part (in the first set of braces), which is independent of $\barn$, i.e., of the residue class of $n$, and a periodic part (in the second set of braces), which depends on $\barn$.  When $n$ is a multiple of $d$, then $\barn=0$ and the equations reduce to the invariant part.  

If the degree is $e$ (which is 3 for $\alpha^{d/c}$ and 4 for $\beta^{d/c}$), the coefficients $\bar\zeta_{e-i}(\barn)$\label{d:barzeta} of $n^{e-i}$ of the periodic part have the alternating symmetry $\bar\zeta_{e-i}(d-\barn) = (-1)^i\bar\zeta_{e-i}(\barn)$.  For instance, in Equation~\eqref{E:alpha d/c} $e=3$ and the periodic part of the coefficient of $n$ (i.e., $i=2$) is $\frac{\barn(d-\barn)}{d^2} (d-c)$, which is invariant under the mapping $\barn \mapsto d-\barn$ (for $1\leq \barn \leq d$).  That is, for any $k \in \bbZ_{>0}$ and any $\barn = 1,2,\ldots,d-1$, $\bar\zeta_i(kd+\barn) = \bar\zeta_i(kd+(d-\barn))$ for all $i$.

The fact that there is no second leading term will be important in examples.

\begin{proof}
The number of attacking pairs is the sum over all lines with slope $d/c$ of $|\lB^{d/c}(b)|^2$.  From Lemma~\ref{p:bL(n)} we can write out the total number:
\begin{align*}
\alpha^{d/c}(n) &= 2cd \sum_{l=1}^{\delta-1} l^2 + [cd+c\barn+(d-\barn)(n-c\delta)] \delta^2 + [\barn(n-c\delta)] (\delta+1)^2,
\end{align*}
which simplifies to Equation~\eqref{E:alpha d/c} after eliminating $\delta$ via $\delta = (n-\barn)/d$.

If $c<d$ the period $d$ follows from examining the coefficient of $n$, which equals $c/3$ only when $\barn=0$.  If $c=d$, then both equal 1 and the period is $d=1$.

The computation for attacking triples is similar.  The total number of such triples is 
\begin{align*}
\beta^{d/c}(n) &= 2cd \sum_{l=1}^{\delta-1} l^3 + [cd+c\barn+(d-\barn)(n-c\delta)] \delta^3 + [\barn(n-c\delta)] (\delta+1)^3,
\end{align*}
which simplifies to Equation~\eqref{E:alpha d/c 3}.  The constant term has period exactly $d$.  
The $n^1$ term also has period $d$, because when $c=d$, both $c$ and $d$ must equal $1$ since they are relatively prime.  
\end{proof}

For the piece $\pP$ we have the formula 
\begin{equation}
a_\pP(2;n) = \sum_{(c,d) \in \M} \alpha^{d/c}(n) - (|\M|-1)n^2 ,
\label{E:alphasum}
\end{equation}
which is the sum over all moves $(c,d) \in \M$ of the number of placements of two labelled pieces that attack along that direction, reduced by the overcount of two pieces on the same square, which should be counted only once per square.  By Equation~\eqref{E:2},
\begin{equation}
u_\pP(2;n) = \frac{1}{2!} o_\pP(2;n) = \frac{1}{2} \Big[ n^4 - \sum_{(c,d) \in \M} \alpha^{d/c}(n) + (|\M|-1)n^2 \Big].
\label{E:2a}
\end{equation}
the number of placements of two labelled pieces that do not attack each other.  Then by Proposition~\ref{p:attackdc} we have an explicit formula for $u_\pP(2;n)$.

\begin{thm}\label{T:u2P}
For $(c,d)\in \M$, let $\hatc := \min(|c|,|d|)$, $\hatd := \max(|c|,|d|)$\label{d:cdhat}, and $\barn := (n \mod \hatd) \in \{0,1,\dots,\hatd-1\}$.  On the square board, 
\begin{equation*}
\begin{aligned}
u_\pP(2;n) &= \frac{1}{2!} o_\pP(2;n) \\
&=   \frac{1}{2} n^4 - \frac{1}{6} \sum_{r=1}^{|\M|} \frac{3\hatd_r-\hatc_r}{\hatd_r^2} n^3 + \frac{|\M|-1}{2} n^2 - \frac{1}{6} \sum_{r=1}^{|\M|} \hatc_r n  \\[3pt]
&\quad 
- \frac{1}{2} \sum_{r=1}^{|\M|} \frac{\barn_r(\hatd_r-\barn_r)(\hatd_r-\hatc_r)}{\hatd_r^2} n + \frac{1}{3} \sum_{r=1}^{|\M|} \frac{\hatc_r(\hatd_r-\barn_r)(\hatd_r - 2\barn_r)\barn_r}{\hatd_r^2}  .
\qed
\end{aligned}
\end{equation*}
The period of $u_\pP(2;n)$ is $\Lambda:=\lcm\{\hatd_r: 1\leq r \leq |\M|\}$.  
\end{thm}

\begin{proof}[Proof of the period bound]
Term $r$ in the penultimate summation has period $\hatd_r$; the term equals 0 only when $\barn_r=0$.  Otherwise the term is positive.  It follows that the sum is 0 only when $n$ is divisible by every $\hatd_r$.  The final summation similarly has period dividing $\Lambda$.
\end{proof}

Notice that the three highest terms are independent of the residue class of $n$.

\begin{cor}\label{C:u2Ptypes}
The number of combinatorial types of nonattacking configuration of two pieces is the number of basic moves.
\end{cor}

\begin{proof}
We already proved this geometrically in Proposition~\Ptwopiecetypes; therefore, evaluating $u_\pP(2;-1)$, which is the number of types by Theorem~\Ttypenumber, checks the correctness of Theorem~\ref{T:u2P}.  We omit the computation, noting only that $\barn_r=\hatd_r-1$ and the result is $|\M|$, as it should be.
\end{proof}

\sectionpage\section{Coefficients of Subspace Ehrhart Functions}\label{ehrhartcoeffs}

The more we can say about the open Ehrhart quasipolynomials $\alpha(\cU;n)$ of subspaces, the more we can infer about the configuration counting functions.

\subsection{Odd and even functions}\label{oddeven}\

The square board has a property that few other boards share (an isosceles right triangle with a side parallel to an axis being one that does).  We remind the reader that a function $f(n)$ of an integer $n$ is called \emph{even} or \emph{odd} if it satisfies $f(-n)=f(n)$ or, respectively, $f(-n)=-f(n)$.  

\begin{thm}[Parity Theorem]\label{T:parity}
Consider the square board with any piece $\pP$.  Let $\cU\in\cL(\cA_\pP)$.  The function $\alpha(\cU;n)$ is an odd function of $n$ if $\dim\cU$ is odd and an even function if $\dim\cU$ is even.
\end{thm}

\begin{proof}
The crucial property of the square board that makes the theorem true is that the interior lattice points, $(n+1)(0,1)^2\cap\bbZ^2$, are isomorphic to all the lattice points, $(n-1)[0,1]^2\cap\bbZ^2$, by a translation.\footnote{A polytope with this property is called Gorenstein with index 2.  (We thank a referee for this observation.)  Any Gorenstein board with index 2, such as the aforementioned right triangle, will satisfy the Parity Theorem.}  The equations of the attack hyperplanes are invariant under that translation, so the two sets of lattice points are equivalent for the inside-out Ehrhart theory of $(\cube,\cA_\pP)$.  Therefore, 
$$
\alpha(\cU;n) := E_{\ocube\cap\tcU}(n+1) = E_{\cube\cap\tcU}(n-1).
$$
Because every $\cU\in\cL(\cA_\pP)$ meets the open hypercube, $\ocube\cap\tcU = (\cube\cap\tcU)^\circ$ and both have dimension $\dim\tcU$.  Therefore, 
$E_{\ocube\cap\tcU}(t) = E_{\cube\cap\tcU}^\circ(t).$
By Ehrhart reciprocity, 
\begin{align*}
E_{\cube\cap\tcU}(n-1) 
&= (-1)^{\dim\tcU} E_{\cube\cap\tcU}^\circ(-(n-1)) 
= (-1)^{\dim\tcU} E_{\cube\cap\tcU}^\circ(-n+1) \\
&= (-1)^{\dim\tcU} E_{\ocube\cap\tcU}(-n+1) = (-1)^{\dim\cU} \alpha(\cU;-n).
\end{align*}
Since $\dim\cU\equiv\dim\tcU \mod 2$, that concludes the proof.
\end{proof}

Oddness or evenness of $\alpha(\cU;n)$ should not be confused with that of its constituent polynomials.  The correct constituent properties are the following.  Let $p(\cU)$ denote the period of $\alpha(\cU;n)$.

\begin{cor}[Constituent Parity]\label{C:parity}
The constituent $\alpha_0(\cU;n)$ is an odd function of $n$ if $\dim\cU$ is odd and an even function if $\dim\cU$ is even.  If $p(\cU)$ is even, the middle constituent $\alpha_{p(\cU)/2}(\cU;n)$ is also odd or even, respectively.  For $i\in [p(\cU)-1]$, there is the relation $\alpha_{p(\cU)-i}(\cU;n) = (-1)^{\dim\cU}\alpha_i(\cU;-n).$
\end{cor}

Thus, if the period is 2, indeed each constituent is an odd or even polynomial, but that is not necessarily so for any larger period.  For example, the contribution of a hyperplane is $\alpha(\cH^{d/c}_{12};n) = \alpha^{d/c}(n)$, given by Equation~\eqref{E:alpha d/c}, with period $d$.  When $d\leq 2$ we get an odd polynomial in $n$ but since $d>2$ gives a nonzero constant term the polynomial is no longer odd.

The preceding corollary can be strengthened by focussing on individual terms.  Let $\cU$ involve $\kappa$ pieces and have codimension $\nu$, and write 
\begin{equation}
\alpha(\cU;n) := \sum_{j=0}^{2\kappa-\nu} \bar\gamma_j(\cU)n^{2\kappa-\nu-j}
\label{E:subcoeffs}
\end{equation}
and define $p_j(\cU)$ to be the (smallest) period of $\bar\gamma_j(\cU)$, which may be less than the period $p(\cU)$ of $\alpha(\cU;n)$ (indeed, the latter equals $\lcm_j p_j(\cU)$).  Thus, $\bar\gamma_j(\cU)$ cycles through the functions $\bar\gamma_{0j}(\cU),\ldots,\bar\gamma_{p_j(\cU),j}(\cU)$.  If $p(\cU)>p_j(\cU)$, then $\bar\gamma_{ij}(\cU)$ cycles through more than one period as $\alpha_i(\cU;n)$ goes through one of its periods.
(Notational note:  We employ $\gamma_i$ for coefficients of an unlabelled counting function $u_\pP(q;n)$ and $\bar\gamma_i$ for a coefficient of a labelled counting function $o_\pP(q;n)$ or $\alpha(\cU;n)$.)

\begin{cor}[Coefficient Parity]\label{C:coeffcycle}
Let $0\leq i<p_j(\cU)$.  The constituents of $\bar\gamma_j(\cU)$ satisfy $\bar\gamma_{p_j(\cU)-i,\,j}(\cU) = (-1)^{\dim\cU-j}\bar\gamma_{ij}(\cU).$

Assume $j \not\equiv \dim\cU \pmod 2$; then $\bar\gamma_{0j}(\cU)=0$ and (if $p_j(\cU)$ is even) $\bar\gamma_{p_j(\cU)/2,\,j}(\cU)=0$.  In particular, if $p_j(\cU)\leq2$ then $\bar\gamma_j(\cU)=0$.  
\end{cor}

Note that $\dim\cU$ can be replaced in these formulas by $\dim\tcU$ since they have the same parity.

\subsection{The second leading coefficient of a subspace Ehrhart function}\label{2ndehrhartcoeff}\

We saw in Proposition~\ref{p:attackdc} that $\alpha(\cU^\kappa;n)$ has no second leading term (the term of $n^{2\kappa-\codim\cU-1}$) when $\cU$ is either a hyperplane $\cH^{d/c}_{ij}$ or a subhyperplane of the form $\cW^{\,d/c}_{ijk}$.  This is a general phenomenon.

\begin{thm}\label{T:nearleading zero}
For every subspace $\cU \in \cL(\cA_\pP)$, the coefficient $\bar\gamma_1(\cU)$ of the second leading term in $\alpha(\cU;n)$ is zero.  
\end{thm}

\begin{proof}
We apply the theorem of McMullen~\cite[Theorem 6]{McM} that in the (open) Ehrhart quasipolynomial $\gamma_0(\cP) n^d + \gamma_1(\cP) n^{d-1} + \cdots + \gamma_d(\cP)$ of a rational convex polytope $\cP$ of dimension $d$, the coefficient $\gamma_i(\cP)$ has period that is a divisor of a quantity $\pi_i$ called the $i$-index of $\cP$,\label{d:pi} defined as the smallest positive integer $\pi$ such that every $(d-i)$-face of $\cP$ contains a rational point (it need not be in lowest terms) with denominator $\pi$.  (This is equivalent to the face's affine span being generated by rational points with denominator $\pi$, which is McMullen's definition.)  Thus, for instance, if every facet of $\cP$ spans an affine flat that contains an integral point, the 1-index of $\cP$ is 1, i.e., $\bar\gamma_1(\cP)$ is constant.

We apply McMullen's theorem to $\cP = \cU \cap [0,1]^{2q}$ where $\cU \in \cL(\cA_\pP)$; that is, $\cP$ is the part of the subspace $\cU$ bounded by the inequalities $x_i,y_i\geq0$ and $x_i,y_i\leq1$ for $i\in[2q]$.  A face of $\cP$ is the restriction of $\cP$ to some subset of boundary hyperplanes; i.e., we fix some set of $x_i$'s and $y_i$'s to be 0 and some other set to be 1.

The equations of $\cU$ have the form $z_j-z_i \perp (d_r,-c_r)$ for some $i<j$ and $(c_r,d_r)\in\M$.  The equation means that $z_j-z_i$ is parallel to $m_r$.  If $\cU$ satisfies $z_j-z_i \parallel m_r$ for more than one basic move $m_r$, then $z_i=z_j$ so we can reduce $q$ by identifying the $i$th and $j$th pieces; therefore we may assume that no pair of pieces appears in more than one equation satisfied by $\cU$.

First, consider $i=1$; that means we choose one $x_i$ or $y_i$ to be 0 or 1.  Let that be $x_1$.  Define $\Delta z := (z,z,\ldots,z)\in\bbR^{2q}$.  All such points with $z\in[0,1]^2$ belong to $\cP$; therefore in particular, each facet $\cP \cap \{x_1=k\}$ (where $k=0$ or 1) contains the integral point $\Delta(k,1)$.  Consequently, the index $\pi_1=1$.  The Parity Theorem~\ref{T:parity} implies then that $\bar\gamma_1(\cU) = 0$.
\end{proof}

\sectionpage\section{The Form of Coefficients on the Square Board}\label{coeffsformsquare}\

On the square board it is possible to get fairly detailed information about the highest-order coefficients of the counting functions of nonattacking configurations.

First we list the exact contributions to $o_\pP(q;n)$ of subspaces with low codimension.  
For $(c,d)\in \M$ let $\hatc=\min(|c|,|d|)$ and $\hatd=\min(|c|,|d|)$ and define $\barn:=(n\mod\hatd)\in\{0,1,\ldots,\hatd-1\}$.  

\begin{lem}\label{L:total01}
The total contribution to $o_\pP(q;n)$ of the subspace of codimension $0$ is $n^{2q}$.

The total contribution of the subspaces of codimension $1$ is $-\binom{q}{2}A_1(n)n^{2q-4}$ where 
\begin{equation}
A_1(n) := \sum_{m\in\M} \alpha^m(n) = a_{10}n^3 + a_{12} n + a_{13}
\label{E:A1}
\end{equation}
with
\begin{align*}
a_{10} &= \sum_{(c,d)\in\M} \frac{3\hatd-\hatc}{3\hatd^2} , \\
a_{12} &= \sum_{(c,d)\in\M} \frac{ \hatc\hatd^2 + 3(\hatd-\hatc) \nhatd(\hatd-\nhatd) }{3\hatd^2} , \\ 
a_{13} &= - \sum_{(c,d)\in\M} \frac{\hatc}{3\hatd^2} \nhatd(\hatd-\nhatd)(\hatd-2\nhatd) .
\end{align*}
The period of $A_1(n)$, as well as that of $a_{12}$, is $\Lambda:=\lcm\{\hatd_r: 1\leq r \leq |\M|\}$.  
\end{lem}

The appearance of $\nhatd$ in $a_{12}$ and $a_{13}$ means that they depend on $n$ through its residue class modulo $\hatd$, unlike $a_{10}$, which is a constant.  

\begin{proof}
The sign in $-A_1(n)$ comes from the fact that the M\"obius function $\mu(\hat0,\cH)=-1$ for a hyperplane $\cH$.  The binomial coefficient counts the number of pairs $\{i,j\}$.  
The evaluation of $A_1$ comes from Proposition~\ref{p:attackdc} and the periods come from its proof.
\end{proof}

\begin{thm}[Square-Board Coefficient Theorem]\label{T:gammapolysquare}
On the square board the coefficients $\gamma_i$ for $i\leq2$ are independent of $n$.  The coefficient $q!\gamma_i$ of $n^{2q-i}$ in $o_\pP(q;n)$ is a polynomial in $q$, of degree $2i$, which depends periodically on $n$.  The leading coefficients are 
\begin{align}
q!\gamma_0 &= 1, \notag\\
q!\gamma_1 &= -(q)_2 \frac{a_{10}}{2} , \notag\\
\intertext{and for $i\geq1$ the coefficients are}
q!\gamma_i &
=: (q)_{2i} \bar\theta_{i,2i} + (q)_{2i-1} \bar\theta_{i,2i-1} + \cdots + (q)_2 \bar\theta_{i,2} \notag
\\&
= \sum_{\kappa=2}^{2i} (q)_\kappa  \sum_{\nu=\lceil\kappa/2\rceil}^{\min(i,2\kappa-2)} \sum_{[\cU_\kappa^\nu]}  \mu(\hat0,\cU_\kappa^\nu) \frac{1}{|\Aut(\cU_\kappa^\nu)|} \, \bar\gamma_{i-\nu}(\cU_\kappa^\nu),
\label{E:gammasumsquare}
\end{align}
in which $\bar\theta_{i,2i}=(-{a_{10}}/{2})^i/i!$, the inner sum ranges over intersection subspace types ${[\cU_\kappa^\nu]}$ such that $\cU_\kappa^\nu \in \cL(\cA_\pP^q)$, and
\[
\bar\theta_{i,2} = \begin{cases}
-a_{10}/2	&\text{ for } i=1, \\
(|\M|-1)/2	&\text{ for } i=2, \\
-a_{12}/2	&\text{ for } i=3, \\
-a_{13}/2	&\text{ for } i=4, \\
0	&\text{ for } i>4.
\end{cases}
\]
The period is a divisor of the least common multiple of the periods of all counting quasipolynomials $\alpha(\cU;n)$ for $\cU \in \cL(\cA_\pP^q)$ such that $\codim\cU \leq i$.  
\end{thm}

\begin{proof}
The polynomiality of $q!\gamma_i$ is part of Theorem~\Tgammapoly, as is the constancy (with respect to $n$) of $\gamma_0$ and $\gamma_1$.  
The value of $\gamma_2$ is determined by subspaces of codimension 2 or less.  The contribution from codimension 2 is independent of $n$ since it is the sum of leading coefficients.  By Equation~\eqref{E:A1}, hyperplanes contribute zero.  The contribution from $\bbR^{2q}$ is zero.  Thus, $\gamma_2$ is independent of $n$.

It remains to investigate the individual coefficients $q!\gamma_i$ more closely.  Because the auxiliary variable $N$\label{d:Naux} now is simply $n^2$, we recalculate the formula for $o_\pP(q;n)$ by rewriting Equation~\Eiopmu, taking account of the simple form $\alpha(\bbR^{2q};n)=n^{2q}$, substituting via Equation~\eqref{E:subcoeffs}, and simplifying as in the proof of Theorem~\Tgammapoly, to get 
\begin{equation}
\begin{aligned}
o_\pP(q;n) 
&= n^{2q} + \sum_{\kappa=2}^{2i} (q)_\kappa  \sum_{\nu=\lceil\kappa/2\rceil}^{\min(i,2\kappa-2)} \sum_{[\cU_\kappa^\nu]}  \mu(\hat0,\cU_\kappa^\nu) \frac{1}{|\Aut(\cU_\kappa^\nu)|} \sum_{j=0}^{2\kappa-\nu} \bar\gamma_j(\cU_\kappa^\nu) n^{2q-\nu-j} ,
\end{aligned}
\label{E:typesum2}
\end{equation}
summed over intersection subspace types $[\cU_\kappa^\nu]$ such that $\cU_\kappa^\nu \in \cL(\cA_\pP^q)$.  So, for $i>0$ we have
\[
q!\gamma_i = \sum_{\kappa=2}^{2i} (q)_\kappa  \sum_{\nu=\lceil\kappa/2\rceil}^{\min(i,2\kappa-2)} \sum_{[\cU_\kappa^\nu]}  \mu(\hat0,\cU_\kappa^\nu) \frac{1}{|\Aut(\cU_\kappa^\nu)|} \bar\gamma_{i-\nu}(\cU_\kappa^\nu) .
\]

The next step is to determine the leading term and $q!\gamma_1$.  To do so we study one isomorphism type of subspace:

\medskip
{\bf Type $\cU_{2i}^i$:}
The subspaces $\cU_{2i}^i$ are those isomorphic to $\cU = \cH_{12}^{l_1} \cap \cdots \cap \cH_{2i-1,2i}^{l_i}$, where $l_1,\ldots,l_i\in \M$.  The M\"obius function is $\prod_j \mu(\hat0,\cH_{2j-1,2j}^{l_j}) = (-1)^i$.  
The automorphisms depend on the selection of slopes.  Write $\M = \{m_1,m_2,\ldots,m_s\}$ where $s=|\M|$.  Suppose $k_r$ hyperplanes have slope $m_r$.  Then an automorphism of $\cU$ can reverse the subscripts in any pair and it can permute the hyperplanes with the same slope.  Thus, $|\Aut(\cU)| = 2^i k_1! \cdots k_s!$.  The value of $\alpha(\cU;n)$ is $\prod_{r=1}^s \alpha^{m_r}(n)^{k_r}$, so the total contribution of all subspaces of type $\cU_{2i}^i$ is
\begin{align*}
\sum_{(k_1,\ldots,k_s)} (-1)^i \frac{1}{2^i k_1! \cdots k_s!} \prod_{r=1}^s \alpha^{m_r}(n)^{k_r} 
&= (-1)^i \frac{1}{2^i i!}  \sum_{(k_1,\ldots,k_s)} \frac{i!}{k_1! \cdots k_s!} \prod_{r=1}^s \alpha^{m_r}(n)^{k_r} \\
&= (-1)^i \frac{1}{2^i i!}  A_1(n)^i,
\end{align*}
the sum being taken over all $s$-tuples $(k_1,\ldots,k_s)$ of nonnegative integers whose total is $s$.  Since the leading coefficient of $A_1(n)$ is $a_{10}$, the coefficient of $(q)_{2i}$ in Equation~\eqref{E:gammasumsquare} is $-a_{10}/2$.  Since we assumed $i>0$, that implies the complete formula for $q!\gamma_1$.
\medskip

The term of $(q)_2$ appears only when $2 \geq i/2$, i.e., $i\leq4$.  The subspaces can be $\cU_2^1$, i.e., hyperplanes, and $\cU_2^2$, i.e., $\cW^{\,=}_{12}$ and its isomorphs.  ($\cW^{\,=}_{12}$\label{d:W=} is the subspace of configurations in which $z_1=z_2$.)

\medskip
{\bf Type $\cU_2^1$:}
Every hyperplane has $\mu(\hat0,\cU)=-1$ and $|\!\Aut\cU|=2$.  Thus, the contribution of all hyperplanes is $-A_1(n)/2$.  Denoting (as before) the coefficient of $n^{3-j}$ by $a_{1j}$, the hyperplanes contribute $-a_{1,i-1}/2$ to $\bar\theta_{i2}$.  That is zero when $i=2$.

\medskip
{\bf Type $\cU_2^2$:}
We have $\mu(\hat0,\cW^{\,=}_{12})=|\M|-1$ and $|\!\Aut\cW^{\,=}_{12}|=2$.  Since $\alpha(\cW^{\,=}_{12};n)=n^2$, the contribution of $[\cW^{\,=}_{12}]$ to $\bar\theta_{i2}$ is $(|\M|-1)\bar\gamma_{2-2}(\cW^{\,=}_{12})/2 = (|\M|-1)/2$ when $i=2$, and otherwise zero.
\end{proof}

If we could obtain the coefficient $\bar\theta_{i,2i-1}$ of $(q)_{2i-1}$ we would have, in particular, the missing coefficient $\bar\theta_{23}$ of a general formula for $q!\gamma_2$.  There is but one difficult step in that.  Since $\nu=i$, 
$$
\bar\theta_{i,2i-1} = \sum_{[\cU]:\cU=\cU_{2i-1}^i}  \mu(\hat0,\cU) \frac{1}{|\Aut(\cU)|} \, \bar\gamma_0(\cU).
$$
The subspaces $\cU_{2i-1}^i$ have the form $\big(\cH_{12}^{l_1}\cap\cH_{23}^{l_2}\big) \cap \cH_{45}^{l_3} \cap \cdots \cap \cH_{2i-2,2i-1}^{l_i}$, where $l_1,\ldots,l_i\in \M$.  There are two types: $l_1=l_2$ and $l_1 \neq l_2$.  The automorphism group has order $H 2^{i-2}$ where $H:=|\Aut\big(\cH_{12}^{l_1}\cap\cH_{23}^{l_2}\big)|.$\label{d:H} The contribution of all subspaces of either type is 
$$
\frac{\mu(\hat0,\cH_{12}^{l_1}\cap\cH_{23}^{l_2})}{H 2^{i-2}}\alpha(\cH_{12}^{l_1}\cap\cH_{23}^{l_2};n)\cdot(-1)^{i-2}A_1(n)^{i-2},
$$
whose leading coefficient is 
$$
\frac{\mu(\hat0,\cH_{12}^{l_1}\cap\cH_{23}^{l_2})}{H}\big(-\frac{a_{10}}{2}\big)^{i-2}\bar\gamma_0(\cH_{12}^{l_1}\cap\cH_{23}^{l_2}).
$$
So we need the value of $\bar\gamma_0(\cH_{12}^{l_1}\cap\cH_{23}^{l_2})$ summed over all permitted slope pairs $(l_1,l_2)$.
\medskip

{\bf Type $\cU_{2i-1\mathrm{a}}^i$:}  
If $l_1=l_2$, then 
$\cH_{12}^{l_1}\cap\cH_{23}^{l_2} = \cW^{\,l_1}_{123},$ 
$\mu(\hat0,\cW^{\,l_1}_{123})=2$, $H=3!$, and $\sum_{l_1}\bar\gamma_0(\cW^{\,l_1}_{123}) = b_{10} := \sum_{(c,d)\in\M} (2\hatd-\hatc)/2\hatd^3$\label{d:b10} (from Proposition~\ref{p:attackdc}).
The total contribution of this type to the coefficient is therefore 
$$
(-1)^i \frac{1}{3\cdot2^{i-2}(i-2)!} \, a_{10}^{i-2} b_{10} . 
$$ 

{\bf Type $\cU_{2i-1\mathrm{b}}^i$:}
When $l_1 \ne l_2$ we have $\mu(\hat0,\cH_{12}^{l_1}\cap\cH_{23}^{l_2})=1$ and $H=1$, but $\alpha(\cH_{12}^{l_1} \cap \cH_{23}^{l_2};n)$ for  arbitrary slopes $l_1$ and $l_2$ is too complicated for us.  
Finding just its leading coefficient would give $\bar\theta_{i,2i-1}$.
\medskip

\sectionpage\section{One-Move Riders }\label{1move}\

Theorem~\ref{T:gammapolysquare} is best possible for pieces in general.  
Even for a piece with only one attacking move, the coefficient $\gamma_3$ may vary with $n$ with a large period.  We prove that here (without appealing to the general theory).  This leads us to propose that periodic variability of higher quasipolynomial coefficients occurs, not due to the number of attacking moves, but because of their slopes.  

The intersection lattice $\cL(\cA_\pP)$ of a one-move rider $\pP$ is the partition lattice $\Pi_q$,\label{d:Piq} so the M\"obius function is known.

Consider a piece $\pP$ with move set $\M=\{(c,d)\}$, where $c$ and $d$ are relatively prime integers such that $0 \leq c \leq d$ and $d>0$.  Note that a move with $c=0$ must have $d=\pm1$, by nontriviality (the zero move is not allowed) and relative primality.

\begin{prop}\label{P:period cd}
For a piece $\pP$ with move set\/ $\M = \{(c,d)\}$ where $0 \leq c \leq d$, 
\begin{align*}
u_\pP(1;n) &= n^2, 
\\%
u_\pP(2;n) &= \frac12 \bigg\{ n^4 + \frac{c-3d}{3d^2}  n^3 - \frac{c}{3} n\bigg\} 
\\&\quad
+\frac{\barn(d-\barn)}{6d^2}\bigg\{  [3(c-d)]n + c(d-2\barn) \bigg\}, 
\\
u_\pP(3;n) &=
\frac{1}{6} \bigg\{n^6 + \frac{c-3d}{d^2} n^5 - \frac{c-2d}{d^3} n^4 
-c n^3 + \frac{c}{d} n^2  \bigg\}
\\&\quad
+ \frac{\barn(d-\barn)}{6d^3} \bigg\{ 
 \big[ 3d(c-d)  \big] n^3
+ \big[ 6d+cd^2-2cd\barn-6c \big] n^2
\\&\qquad\qquad\qquad\ 
+  \big[ 2(d-2 c)(d - 2\barn)\big] n
+  3c(d-\barn)\barn  
\bigg\} ,
\\
u_\pP(4;n) &=
\frac{1}{24} \bigg\{n^8 + \frac{2(c-3d)}{d^2} n^7 + \frac{c^2-18cd+33d^2}{3d^4} n^6 + \frac{18c-30d-10cd^4}{5d^4} n^5 
\\&\qquad\quad
+ \frac{18cd-2c^2}{3d^2} n^4 -\frac{4c}{d^2} n^3 + \frac{c^2}{3} n^2+\frac{2c}{5}n  \bigg\}
\\&\quad
+ \frac{\barn(d-\barn)}{360d^4} \bigg\{ 
 \big[90 d^2 (c - d) \big] n^5
+ \big[ 30 (c^2 + 15 d^2 -16c d + c d^3 - 2 c d^2 \barn) \big] n^4
\\&\qquad\ 
+ 10 \big[6 d (-9 + 2 d (d - 2 \barn)) + c^2 (d - 2 \barn) + 
   27 c (2 - d^2 + 2 d \barn) \big] n^3
\\&\qquad\ 
+ 15 \big[ 2 d (-12 d + c (18 - c d + d^2)) + 
 3 (-24 c + (16 + c^2) d + 2 c d^2 + d^3) \barn 
\\&\qquad\qquad\quad 
 - 3 (c + d)^2 \barn^2 \big] n^2
\\&\qquad\ 
+  10 \big[ 9 cd^2 - 9d^3-c^2 d^3 
   - (27 d^2 - 81cd + 5 c^2 d^2 - 3 cd^3) \barn 
\\&\qquad\qquad\quad 
   + 9 (3 d + c (cd - d^2-9)) \barn^2 + 6 c ( d-c) \barn^3 \big] n
\\&\qquad\ 
+  c (d - 2 \barn) \big[ d^2 (-6 + 5 c \barn) - 3 d \barn (36 + 5 c \barn) + 
   2 \barn^2 (54 + 5 c \barn) \big]  
\bigg\} .
\end{align*}
For $q=2,3,4$ the period of $u_\pP(q;n)$ with respect to $n$ is $d$.  
\end{prop}

\begin{proof}
The formula for $u_\pP(1;n)$ is trivial.  Theorem~\ref{T:u2P} implies the value of $u_\pP(2;n)$; a combinatorial count similar to that for $u_\pP(3;n)$ and $u_\pP(4;n)$ gives the same result.

Direct combinatorial arguments for $q=3$ and $4$ give 
\begin{align*}
u_\pP(3;n) &=\binom{n^2}{3} - \sum_{l\in \bL(n)} \binom{l}{3}-\sum_{l\in\bL(n)} \binom{l}{2}\big[n^2-l\big] \qquad\textup{and} \\
u_\pP(4;n) &= \binom{n^2}{4} - \sum_{l\in \bL(n)} \binom{l}{4}-\sum_{l\in\bL(n)} \binom{l}{3}\big[n^2-l\big]-\sum_{\{l,l'\}\subseteq\bL(n)} \binom{l}{2}\binom{l'}{2} \\
&\quad\ -\sum_{l\in \bL(n)} \binom{l}{2}\bigg[\binom{n^2-l}{2}-\sum_{l'\in \bL(n)}\binom{l'}{2} + \binom{l}{2} \bigg],
\end{align*}
where $\bL(n) := \bL^{d/c}(n)$.  
For instance, $u_\pP(4;n)$ is the number of placements of four non-attacking pieces, which we count by placing four pieces on any of the $n^2$ positions on the board and removing those where at least two pieces attack.  We must remove the cases where four pieces are in the same line $l^{d/c}(b)$, those where three pieces are in the same line and the fourth is in another line, those where two pieces are in the same line $l^{d/c}(b)$ and the remaining two are both in another line $l^{d/c}(b')$, and last, those where two pieces are attacking and the remaining two pieces attack none of the others.

The reasoning for $u_\pP(3;n)$ is simpler so we merely show the steps in the simplification:
\begin{align*}
u_\pP(3;n) &= 
\binom{n^2}{3} - \sum_{l\in \bL(n)} \binom{l}{3}-\sum_{l\in\bL(n)} \binom{l}{2}\big[n^2-l\big] 
\\&= 
\binom{n^2}{3} +2\sum_{l\in \bL(n)} \binom{l}{3}-(n^2-2)\sum_{l\in\bL(n)} \binom{l}{2} 
\\&= 
\binom{n^2}{3} 
+ 2 \bigg\{ 2cd \binom{\frac{n-\barn}{d}}{4} + \Big[ (d-\barn)\big(n-c\frac{n-\barn}{d}\big)+c(\barn+d) \Big] \binom{\frac{n-\barn}{d}}{3} 
\\&\qquad\qquad\qquad
+ \Big[ \barn\big(n-c\frac{n-\barn}{d}\big) \Big] \binom{\frac{n-\barn}{d}+1}{3} 
 \bigg\} 
\\&\qquad
- (n^2-2) \bigg\{    2cd \binom{\frac{n-\barn}{d}}{3} + \Big[ (d-\barn)\big(n-c\frac{n-\barn}{d}\big)+c(\barn+d) \Big] \binom{\frac{n-\barn}{d}}{2} 
\\&\qquad\qquad\qquad\quad
+ \Big[ \barn\big(n-c\frac{n-\barn}{d}\big) \Big] \binom{\frac{n-\barn}{d}+1}{2}    \bigg\} ,
\end{align*}
which when expanded (we used Mathematica) gives the constant and periodic parts stated in the proposition.  The simplification of $u_\pP(4;n)$ is similar.

The constant terms of the quasipolynomials for $q=2,3,4$ have period $d$ since they are zero only when $\barn=0$.  (Recall that $c>0$.)  The other coefficients have period dividing $d$ since they depend on $n$ through $\barn$.
\end{proof}

The equation for $u_\pP(3;n)$ agrees with the formulas for partial queens $\pQ^{10}$ and $\pQ^{01}$ in Part~III.
It would be instructive to find $u_\pP(q;n)$ in general, but this task seems difficult.

The example with $(c,d)=(1,2)$ simplifies nicely when $q = 2, 3, 4$.

\begin{cor}\label{prop:period12}
For a piece $\pP$ with move set\/ $\M=\{(1,2)\}$, the following formulas hold.
\begin{align*}
u_\pP(2;n) &= \left\{\frac{n^4}{2}-\frac{5 n^3}{24}-\frac{11 n}{48}\right\}+(-1)^n\frac{n}{16},
\\[5pt]
u_\pP(3;n) &= \left\{\frac{n^6}{6}-\frac{5 n^5}{24}+\frac{n^4}{16}-\frac{11 n^3}{48}+\frac{7 n^2}{48}+\frac{1}{32}\right\}+(-1)^n\left\{\frac{n^3}{16}-\frac{n^2}{16}-\frac{1}{32} \right\} 
\\[5pt]
u_\pP(4;n) &= \left\{\frac{n^8}{24}-\frac{5 n^7}{48}+\frac{97 n^6}{1152}-\frac{131 n^5}{960}+\frac{223 n^4}{1152}-\frac{17 n^3}{192}+\frac{137 n^2}{2304}-\frac{73 n}{1920}\right\}
\\&\qquad+(-1)^n\left\{\frac{n^5}{32}-\frac{29 n^4}{384}+\frac{3 n^3}{64}-\frac{35 n^2}{768}+\frac{7 n}{128}\right\}.
\end{align*}
\end{cor}

In each formula of Proposition~\ref{P:period cd} the coefficient $\gamma_3$ and the entire formula both have period $d$.  This suggests generalizations.

\begin{prop}\label{P:2q-3 period1}
For a one-move rider with basic move $(c,d)$ where $0 \leq c \leq d$, the period of $\gamma_3$ in $u_\pP(q;n)$ when $q\geq2$ is exactly $d$.  The periodic part is 
$$
-\frac{1}{(q-2)!} \frac{\barn(d-\barn)(d-c)}{2d^2}.
$$
\end{prop}

\begin{proof}
The subspaces that can contribute to $q!\gamma_3$ are those of codimension at most 3.  There is no contribution from codimension 0 since the open Ehrhart quasipolynomial is $\alpha(\bbR^{2q};n)n^{2q}=n^{2q}$.  The contribution from codimension 3 is constant.  

The contribution from a hyperplane is the $n^{2q-3}$ term of $\mu(\hat0,\cH^{d/c}_{ij}) \alpha(\cH^{d/c}_{ij}) n^{2q-4} = -\alpha^{d/c}n^{2q-4}$.  By Equation~\eqref{E:alpha d/c}, the periodic contribution is $-\barn(d-\barn)(d-c)/d^2$.  There are $\binom{q}{2}$ hyperplanes.

There are two kinds of subspace of codimension 2.

\begin{description}

\item[{\bf Type $\cU^2_{4^*}$\,}]
This subspace is the intersection of two hyperplanes that involve disjoint pairs of pieces; i.e., $\cU^2_{4^*} = \cH^{d/c}_{ij} \cap \cH^{d/c}_{kl}$ where $\{i,j\} \cap \{k,l\} = \eset$.  The value of $\alpha(\cU^2_{4^*})$ is $(\alpha^{d/c})^2$, whose $n^5$ term, as one can see from Equation~\eqref{E:alpha d/c}, has coefficient 0.  The contribution of this subspace to the $n^{2q-3}$ term of $o_\pP(q;n)=q!u_\pP(q;n)$ is the corresponding term in $n^{2q-8}\alpha(\cU^2_{4^*})$, so there is no contribution.

\item[{\bf Type $\cW^{\,d/c}_{12}$\,}]
The contribution of each subspace is $\beta^{d/c}$.  Each subspace contributes $\mu(\hat0,\cW^{\,d/c}_{12}) \alpha(\cW^{\,d/c}_{12}) n^{2q-6}$ to $o_\pP(q;n)$, in which the coefficient of $n^{2q-3}$ is 0.

\end{description}

Thus, there is no contribution except from hyperplanes, whose total periodic contribution to $q!\gamma_3$ is easily seen.  Dividing by $q!$ gives the periodic part of $\gamma_3$.

To prove the period is $d$, note that the periodic part vanishes if and only if $\barn=0$ or $c=d$.  In the latter case $c=d=1$ so the period is $d=1$.
\end{proof}

\begin{conj}\label{Cj:period1move}
For a one-move rider with basic move $(c,d)$, the period of $u_\pP(q;n)$ is exactly $\max(|c|,|d|)$.
\end{conj}

The period certainly is a multiple of $\max(|c|,|d|)$ because of the period of $\gamma_3$.

The number of combinatorial types is obviously $1$ (as stated in Theorem~\Ttwomovetypes).  
This implies a check on any formula for $u_\pP(q;n)$, since $u_\pP(q;-1)$ must equal 1.  Applying the check to $u_\pP(2;n)$, $u_\pP(3;n)$, and $u_\pP(4;n)$ in Proposition~\ref{P:period cd}, realizing that $\barn=d-1$, does give $u_\pP(q;-1)=1$ for $q=2,3,4$.

\sectionpage\section{A Formula for the $n$-Queens Problem}\label{nqueens}\

Theorem~\ref{T:gammapolysquare} covers any number of pieces on any size board.  By setting $q=n$ we obtain what can be regarded as the first closed-form formula (according to~\cite{Bell}) for the $n$-Queens Problem, which is the case in which $\pP$ is the queen in the following result.  Let $\cA_\pP^\infty$ be the arrangement in the countably-infinite-dimensional vector space $\bbR^\infty$ of all move hyperplanes $\cH^{d/c}_{ij}$, $\{i,j\} \subset \bbZ_{>0}$.

\begin{thm}\label{T:nqueens}
The number of ways to place $n$ unlabelled copies of a rider piece $\pP$ on an $n\times n$ board so that none attacks another is
\begin{align*}
u_\pP(n;n) 
&= \frac{1}{n!}  \sum_{i=1}^{2n} n^{2n-i} \sum_{\kappa=2}^{2i} (n)_\kappa  \sum_{\nu=\lceil\kappa/2\rceil}^{\min(i,2\kappa-2)} 
\sum_{[\cU_\kappa^\nu]:\cU_\kappa^\nu\in\cL(\cA_\pP^\infty)}  
\mu(\hat0,\cU_\kappa^\nu) \frac{1}{|\Aut(\cU_\kappa^\nu)|} \, \bar\gamma_{i-\nu}(\cU_\kappa^\nu).
\end{align*}
\end{thm}

This formula is very complicated and potentially infinite (potentially rather than actually, because for each value of $n$ the number of nonzero terms is finite) but it is explicitly computable.  We have not tried to compare its complexity with that of other methods of counting nonattacking placements.

\sectionpage\section{Questions, Extensions}\label{last}

Work on nonattacking chess placements raises many questions, several of which have general interest.  Besides Conjecture~\ref{Cj:period1move} and others to appear in later parts, we propose the following directions for research.

\subsection{Detailed improvements}\label{improve}\

These problems concern significant loose ends we left in basic counting questions.

\begin{enumerate}[(a)]
\item Generalize Proposition~\ref{p:attackdc} by finding a formula for the number of ways to place $q$ mutually attacking pieces on the same slope line.  The starting point is that the number of such placements in a line of length $l$ is $l^q$, which would be summed.  A consequence by inclusion--exclusion will be a complete solution for one-move pieces, which in turn may suggest general results about periods.
\label{pattackc-gen}

\smallskip
\item Extend the formulas for $u_\pP(q;n)$ for $q\leq4$ for a general one-move rider (Section~\ref{1move}) to larger numbers of pieces.  This should give more indications of the behavior of periods.

\smallskip
\item Evaluate the coefficient of $(q)_3$ in $q!\gamma_2$ for an arbitrary rider  in Theorem~\ref{T:gammapolysquare} to get a complete formula for $\gamma_2$.

\smallskip
\item It should be feasible to find an explicit formula for three pieces, similar to that for two pieces in Theorem~\ref{T:u2P}.  It would require a solution to Problem~\ref{improve}(a) for $q=4$.  There is one really new behavior: a subspace of codimension 3 may be given by three slope hyperplanes of different slopes on three pieces, that is, of type $\cU^3_3 = \cH^{d/c}_{ij} \cap \cH^{d'/c'}_{jk} \cap \cH^{d''/c''}_{ik}$; finding $\alpha(\cU)$ for $\cU^3_3$ looks harder than for the subspaces solved in Section~\ref{1or2}.  (Since $\cU^3_3$ does not exist for a 2-move rider, 2-move riders could be the first to work on.)
\end{enumerate}

\subsection{Recurrences and their lengths}\label{recurrence}\

\Kot\ obtained empirical formulas for $u_\pP(q;n)$ for relatively large numbers $q$ of various pieces (queen, bishop, nightrider, et al.) by computing the values for $n=1,2,\ldots,N$ where $N$ is fairly large, and looking for a heuristic recurrence relation.  He derives a generating function from that recurrence, then uses the generating function to get a quasipolynomial formula.  Since the recurrence is heuristic, the formula is unproved.  To prove his formula, if the period is $p$, he has to compute up to about $N=2pq$, because the degree of the quasipolynomial being $2q$, there are $2pq$ undetermined coefficients (the leading coefficient being known).  Worse, the period $p$ is unknown.  But if the recurrence is much shorter than $p$, he will find a recurrence (without proof) from a much smaller value of $N$.  That seems always to be the case if $p$ is large.  In other words, there seems to be a recurrence for $u_\pP(q;n)$ that is much shorter than $p$.  Can our method explain this?

The length of the recurrence is the degree of the denominator of the generating function of $u_\pP(n)$ when it is reduced to lowest terms.  
The explanation of a (relatively) short recurrence is that the generating function, which has the standard form $f(x)/(1-x^p)^{2q}$ where $f(x)$ is a polynomial and $p$ is the period, is not in lowest terms. 
Thus, there seems to be a systematic common factor of the numerator and denominator when expressed in standard form.  (One instance is \Kot's conjecture about the denominator for $q$ queens \cite[2nd ed., p.\ 14; 6th ed., p.\ 22]{ChMath}.)  Essentially nothing is known about the presumed common factor, starting with why it exists.  This seems the most important research problem in the subject.

\subsection{Period bounds}\label{bounds}\

As we saw in Section~\ref{recurrence} and will see throughout this series, periods and period bounds (which usually are denominators of inside-out polytopes) are essential information in obtaining formulas.  
However, the period of the whole quasipolynomial $u_\pP(q;n)$ is not the best indicator of the difficulty of computing $u_\pP(q;n)$. 
If we know the periods $p_i$ of the coefficients $\gamma_i$, the number of unknowns in interpolating the quasipolynomial from data becomes less than the number $2pq$ mentioned in Section~\ref{recurrence}.  Precisely stated, the number of undetermined coefficients in $u_\pP(q;n)$ (for fixed $q$), hence the number of values that are needed to determine all coefficients, is $\sum_{i=1}^{2q} p_i$, which in known examples is much smaller than $2pq$.

The difficulty is that it is a large task to compute the periods of all coefficients.  Thus, we would like to have a simple ``universal'' bound on $p_i$ that depends on $q$, $i$, and the set of moves, and is easy to compute.  To that end we offer a few conjectures of practical or theoretical interest.

\subsubsection{Plausible bounds}\label{plausible}\

For a piece $\pP$ with move set $\M$, the \emph{move range} $\|\pP\|$ is the maximum coordinate magnitude of a basic move, that is, 
$$
\|\pP\|:=\max_{(c_r,d_r)\in\M} \hatd_r, \text{ where } \hatd_r:=\max(|c_r|, |d_r|).
\label{d:range}
$$  
For each positive integral distance $\lambda$ there is a ``largest'' piece $\pP^{\max}_\lambda$ with move range $\lambda$; its basic move set is $\M_\lambda := \{ (c,d) : |c|, |d| \leq \lambda,\ \gcd(c,d)=1 \}$, which consists of every basic move consistent with having move range $\lambda$.  For instance, $\pP^{\max}_1$ is the queen and $\pP^{\max}_2$ is the combined queen and nightrider.  

\begin{conj}\label{Cj:maxpiece}
Among all pieces with $\|\pP\|\leq\lambda$, $\pP^{\max}_\lambda$ maximizes the period of every coefficient $\gamma_i$.  (We assume here that $q$ is fixed.)
\end{conj}

\begin{prob}\label{Pr:maxpieceperiod}
Find the period of $u_{\pP^{\max}_\lambda}(q;n)$ and those of its coefficients.  Or find reasonably close bounds.
\end{prob}

These questions are surely hard and probably of theoretical interest only.  In particular, \Kot\ conjectures \cite[sixth ed., p.\ 31]{ChMath} that the period of the queen, $\pQ=\pP^{\max}_1$, is $\lcm(1,2,\ldots,F_q)$, $F_q$ being the Fibonacci number.  The period of the nightrider $\pN$, which has move range $\|\pN\|=2$, grows hugely with $q$.  It seems probable that for each $\lambda$ the period of $\pP^{\max}_\lambda$ grows extremely rapidly with $q$ (and also with $\lambda$), but if \Kot\ is right, it may follow a discernible pattern.  If so, that pattern may be of use.  

\subsubsection{Observed bounds}\label{observed}\

What we most want, though, is a simple formula in terms of, say, $q$ and $\|\pP\|$ that gives an upper bound on the period $p$, or better, which is guaranteed to have $p$ as a divisor.  We have no conjecture about this, but we propose a low bound on the periods of the highest nonconstant coefficients.
Define 
$$
\Lambda:=\lcm\{\hatd_r : m_r \in \M \}.
\label{d:Lambda}
$$

\begin{conj}\label{Cj:gamma3period}
The period of $\gamma_3$ is $1$ or $\Lambda$.  (We see period $\Lambda$ for two pieces---see Theorem~\ref{T:u2P}; for any number of one-move riders---see Proposition~\ref{P:2q-3 period1}; for three partial queens---see Part~III; and for two partial nightriders---see Section~\N.)
\end{conj}

\begin{conj}\label{Cj:gamma4period}
The period of $\gamma_4$ divides $\Lambda$.  
\end{conj}

A second kind of maximal piece, suggested by Conjectures~\ref{Cj:gamma3period} and \ref{Cj:gamma4period}, is $\pP^{\lcm}_\lambda$, whose move set consists of all moves $(c,d)$ such that $\hatd|\lambda$.  It might be called ``lcm-maximal''.  Conceivably it may be more natural than $\pP^{\max}_\lambda$ for bounding periods.

\begin{conj}\label{Cj:lcmpiece}
Among all pieces with $\lcm\{\hatd_r : m_r \in \M \}=\lambda$, $\pP^{\lcm}_\lambda$ maximizes the period of every coefficient $\gamma_i$.  (We assume here that $q$ is fixed.)
\end{conj}

\begin{prob}\label{Pr:lcmpieceperiod}
Find the period of $u_{\pP^{\lcm}_\lambda}(q;n)$ and those of its coefficients.  Or find reasonably close bounds.
\end{prob}

\subsubsection{The geometry of coefficient periods}\

Despite the hopes expressed in the preceding conjectures, the most effective way to bound periods may be to find subspace denominators by geometrical computation.  Geometry also suggests a general monotonicity property. 

We see in formulas here and in Parts~III (Theorem~\Phvdiag) and IV and in \Kot's book \cite{ChMath} that the period $p_i$ of $\gamma_i$ tends to increase with $i$.  The geometry suggests that should be a general truth.  
Let $\cL_0^i = \{ \cU \in \cL(\cA_\pP) : \codim\cU\leq i\}$.  Each $\gamma_i$ depends on the subspaces in $\cL_0^i$.  
Define the denominator $D_i^q$ of the system $(\cB^q,\cL_0^i)$ to be the least common denominator of coordinates of all points determined by intersecting a subspace $\cU\in\cL_0^i$ with the boundary of $\cube$.  Every point so determined for $i$ is also so determined for $i+1$, since $\cL_0^i \subseteq \cL_0^{i+1}.$  It follows that each $D_{i+1}^q$ is a multiple of $D_i^q$.  By Ehrhart theory $p_i$ divides $D_i^q$; we expect $p_i$ to increase weakly with $i$.

\begin{conj}\label{Cj:periodincrease}
The periods $p_i$ of $u_\pP(q;n)$ are weakly monotonically increasing: $1 = p_0 = p_1 \leq \cdots \leq p_{2q}$.  (If so, then the whole period $p=p_{2q}$.)  More precisely, $p_i$ divides $p_{i+1}$.
\end{conj}

And here is a final, stronger conjecture.  In Ehrhart theory in general the period need not equal the denominator, but in our formulas we always find equality.

\begin{conj}\label{Cj:perioddenom}
The period $p$ of $u_\pP(q;n)$ equals the denominator $D(\cube,\cA_\pP)$.  The period $p_i$ of $\gamma_i$ in $u_\pP(q;n)$ equals $D_i^q$.
\end{conj}

\sectionpage\section*{Appendix:  Dictionary of Notation} 

This dictionary refers to the initial definition of the notation in this article, where applicable.  The reader may wish to refer to the dictionary of notation from Part~I as well. 


\noindent\begin{tabular}{ll}

\\ \hbox to 4.5cm{$a_{1i}$}	&coefficients of $A_1(n)$ (p.\ \pageref{E:A1})
\\ $a_\pP(2;n)$	&\# of attacking configurations (p.\ \pageref{d:aP})
\\ $b$		&$y$-intercept of $l^{d/c}(b)$ (p.\ \pageref{d:linesizes})
\\ $(c,d),(c_r,d_r)$	&coordinates of basic move (p.\ \pageref{d:cd2})
\\ $d/c$ \quad\dotfill		&slope of a line (p.\ \pageref{d:cd2})
\\ $(\hatc,\hatd)$	&$(\min,\max)$ of $c,d$ (p.\ \pageref{d:cdhat})
\\ $l$		&index for $\bL^{d/c}(n)$; i.e., line size (p.\ \pageref{d:n/d})
\\ $l^{d/c}(b)$	&line of slope $d/c$, $y$-intercept $b$ (p.\ \pageref{ldcb})
\\ $\lB^{d/c}(b)$ 	&$= l^{d/c}(b) \cap [n]^2$ (p.\ \pageref{ldcb})
\\ $m_r = (c_r,d_r), m=(c,d)$	&basic move (p.\ \pageref{d:cd2})
\\ $n$		&size of square board (p.\ \pageref{d:n}) 
\\ $n+1$		&dilation factor for board (p.\ \pageref{d:n+1}) 
\\ $[n]$ 		&$= \{1,\hdots,n\}$ (p.\ \pageref{d:[n]})
\\ $[n]^2$		&square board (p.\ \pageref{d:[n]2}) 
\\ $\barn$	\quad\dotfill	&$n \mod \hatd$ (p.\ \pageref{d:barn})
\\ $o_\pP(q;n)$	&\# of nonattacking labeled configurations (p.\  \pageref{d:distattacks})
\\ $p$		&period of counting quasipolynomial (p.\ \pageref{d:p})
\\ $p(\cU)$		&period of quasipolynomial $\alpha(\cU;n)$ (p.\ \pageref{d:p})
\\ $p_j(\cU)$	&period of coefficient $\bar\gamma_j(\cU)$ (p.\ \pageref{d:p})
\\ $q$ \quad\dotfill	&\# of pieces on a board (p.\ \pageref{d:P})
\\ $r$		&move index 
\\ $u_\pP(q;n)$	&\# of nonattacking unlabeled configurations (p.\ \pageref{d:indistattacks})
\\ $z=(x,y), z_i=(x_i,y_i)$	&piece position (p.\ \pageref{d:Z})
\\

\end{tabular}

\noindent\begin{tabular}{ll}

\\ \hbox to 4.5cm{$\alpha(\cU;n)$}	&quasipolynomial for a subspace (p.\ \pageref{d:acU})
\\ $\alpha_i(\cU;n)$	&constituent of $\alpha(\cU;n)$ (p.\ \pageref{d:acU})
\\ $\alpha^{d/c}(n)$	&\# of 2-piece collinear attacks (p.\ \pageref{d:adc})
\\ $\beta^{d/c}(n)$	&\# of 3-piece collinear attacks (p.\ \pageref{d:bdc})
\\ $\gamma_i$ \quad\dotfill	&coefficient in unlabeled counting function $u_\pP$ (p.\ \pageref{d:gamma})
\\ $\bar\gamma_j(\cU)$	&coefficient in $\alpha(\cU;n)$ (p.\ \pageref{E:subcoeffs})
\\ $\delta_{ij}$	&Kronecker delta (p.\ \pageref{d:KD})
\\ $\delta$ 		&$= \lfloor n/d\rfloor$ (p.\ \pageref{d:n/d})
\\ $\bar\zeta_i$	&coefficient in periodic part (p.\ \pageref{d:barzeta})
\\ $\bar\theta_{i,\kappa}$ \quad\dotfill &coefficient of $(q)_\kappa$ in $q!\gamma_i$ (p.\ \pageref{E:gammasumsquare})
\\ $\kappa$	&\# of pieces involved in a subspace (p.\ \pageref{d:cU})
\\ $\mu$		&M\"obius function of $\cL(\cA)$ (p.\ \pageref{d:mu})
\\ $\nu$		&codimension of a subspace (p.\ \pageref{d:cU})
\\ $\pi_i$		&$i$-index of $\cP$ (p.\ \pageref{d:pi})
\\

\end{tabular}

\newpage
\noindent\begin{tabular}{ll}

\\ \hbox to 4.5cm{$A_1(n)$}	&\# of attacking ordered pairs (p.\ \pageref{E:A1})
\\ $D$		&denominator of polytope or inside-out polytope (p.\ \pageref{d:gamma})
\\ $E_{\cP}$	&Ehrhart quasipolynomial (p.\ \pageref{d:Ehr})
\\ $E_{\cP}^\circ$ \quad\dotfill	&open Ehrhart quasipolynomial (p.\ \pageref{d:E})
\\ $E_{\cP,\cA}^\circ$	&inside-out open Ehrhart quasipolynomial (p.\ \pageref{d:Eiop})
\\ $H$		&size of automorphism group (p.\ \pageref{d:H})
\\ $I$, $J$, $I_i$, $J_j$	&subsets of $Z$ (p.\ \pageref{d:IJ})
\\ $N$		&auxiliary variable (p.\ \pageref{d:Naux})
\\ $Z$ \quad\dotfill	&lower left border of $[n]^2$ (p.\ \pageref{d:Z})
\\
\\ $\bL^{d/c}(n)$	&multiset of line sizes (p.\ \pageref{d:linesizes})
\\ $\M$	 	&set of basic moves (p.\ \pageref{d:[n]2})
\\
\\ $\cA_{\pP}$	&move arrangement of piece $\pP$ (p.\ \pageref{d:cA})
\\ $\cB, \cBo$	&closed, open board polygon (p.\ \pageref{d:P})
\\ $\cH_{ij}^{d/c}$ 	&hyperplane of move $(c,d)$ (p.\ \pageref{d:cH})
\\ $\cL$		&intersection lattice  (p.\ \pageref{d:L})
\\ $\cP$, $\cP^\circ$ \quad\dotfill	&polytope, open polytope (p.\ \pageref{d:P})
\\ $(\cP,\cA_\pP)$	&inside-out polytope (p.\ \pageref{d:iop})
\\ $\cU$		&subspace in intersection lattice (p.\ \pageref{d:cU})
\\ $[\cU]$		&subspace type (p.\ \pageref{d:type})
\\ $\tcU$		&essential part of $\cU$ (p.\ \pageref{d:cU})
\\ $\cW_{ij\ldots}^{\,d/c}$ \quad\dotfill	&subspace of slope relation (p.\ \pageref{d:Wdc})
\\ $\cW_{ij\ldots}^{\,=}$	&subspace of equal position (p.\ \pageref{d:W=})
\\
\\ $\bbR$		&real numbers
\\ $\bbZ$		&integers
\\
\\ $\pP$		&piece (p.\ \pageref{d:P})
\\ $\|\pP\|$		&move range of piece (p.\ \pageref{d:range})
\\
\\ $\Lambda$	&lcm of move extents (p.\ \pageref{d:Lambda})
\\ $\Pi_q$		&lattice of partitions of $[q]$ (p.\ \pageref{d:Piq})
\\
\\ $\Aut(\cU)$	&subspace automorphism group (p.\ \pageref{d:Aut})
\\ $\codim(\cU)$	&subspace codimension 
\\ $\dim(\cU)$	&subspace dimension
\\

\end{tabular}

\sectionpage
\newcommand\otopu{$\overset{\circ}{\textrm u}$}

\end{document}